\pdfoutput=1
\pdfoutput=1
\pdfoutput=1
\pdfoutput=1
\documentclass[a4paper,12pt]{article} 
\usepackage{calc}
\usepackage[all]{xy}
\usepackage[centertags]{amsmath}
\usepackage{latexsym}
\usepackage{amsfonts}
\usepackage{graphicx}
\usepackage{tikz}
\usepackage{cases}
\usepackage{amssymb}
\usepackage{amsthm}
\usepackage{xcolor}
\usepackage{fancyhdr}
\usepackage[dvips]{epsfig}
\usepackage{newlfont}
\usepackage[latin1]{inputenc}
\usepackage[latin1]{inputenc}
\usepackage[english,french]{babel}
\usepackage{graphicx}
\usepackage{t1enc}
\usepackage[english,french]{babel}
\usepackage{fancybox}
\usepackage{graphicx}
\usepackage{t1enc}
\usepackage[french2]{minitoc}
\usepackage{mathrsfs}
\usepackage{amsfonts}
\usepackage{amssymb}
\usepackage{amsbsy}
\usepackage{amsmath}
\usepackage{wasysym}
\usepackage{hyperref}
\allowdisplaybreaks
\usepackage{float}
\usepackage{geometry}
\usepackage{lineno}
	\numberwithin{equation}{section}
	\theoremstyle{plain}
	
	\newtheorem{lem}{Lemma}

	\newtheorem{rem}{Remark}

	\newtheorem{prop}{Proposition}

	\newcommand{\f}{\Phi}

	\newcommand{\R}{{\mathbb R}}
	\newcommand{\N}{{\mathbb N}}

	\newcommand{\eg}{{{\ell_1} }}
	\newcommand{\ed}{{{\ell_2} }}

	 \makeatletter
	\newcommand{\thechapterwords}
	{ \ifcase \thechapter\or 1\or 2\or 3\or 4\or 5\or
		6\or 7\or 8\or 9\or 10\or 11\fi}
	\def\thickhrulefill{\leavevmode \leaders \hrule height 2ex \hfill \kern \z@}
	\def\@makechapterhead#1{%
		\vspace*{15\p@}%
		{\parindent \z@ \centering \reset@font
			\thickhrulefill\quad
			\scshape  {\chapnumfont \@chapapp{}}{\chapnumfont \thechapterwords}
			\quad \thickhrulefill
			\par\nobreak
			\vspace*{15\p@}%
			\interlinepenalty\@M
			\hrule
			\vspace*{15\p@}%
			\huge {\bfseries  #1}\par\nobreak
			\par
			\vspace*{15\p@}%
			\hrule
			\vskip 15\p@
	}}
	\def\@makeschapterhead#1{%
		\vspace*{15\p@}%
		{\parindent \z@ \centering \reset@font
			\thickhrulefill
			\par\nobreak
			\vspace*{15\p@}%
			\interlinepenalty\@M
			\hrule
			\vspace*{15\p@}%
			\Huge \bfseries #1\par\nobreak
			\par
			\vspace*{15\p@}%
			\hrule
			\vskip 30\p@
	}}
	
	\DeclareFixedFont{\chapnumfont}{T1}{phv}{b}{n}{20pt}
	\DeclareFixedFont{\chapchapfont}{T1}{phv}{b}{n}{16pt}
	\DeclareFixedFont{\chaptitfont}{T1}{phv}{b}{n}{24.88pt}
	\def\@makechapterhead#1{%
		\vspace*{15\p@}%
		{\parindent \z@ \centering \reset@font
			\thickhrulefill\quad
			\scshape {\chaptitfont\color[rgb]{0.00,0.50,1.00}\@chapapp{}}
			{\chapnumfont \thechapterwords}
			\quad \thickhrulefill
			\par\nobreak
			\vspace*{15\p@}%
			\interlinepenalty\@M
			\hrule
			\vspace*{15\p@}%
			{\Large\bfseries #1}\par\nobreak
			\par
			\vspace*{15\p@}%
			\hrule
			\vskip 30\p@
	}}%
	
\begin{document}

 				\selectlanguage{english}
% 				Bi-Lipschitz Homeomorphism and
 	\title{$C^1$-Diffeomorphism  Class of some Circle Maps with a Flat Interval}
 		\author{Bertuel TANGUE NDAWA and Carlos OGOUYANDJOU 
% 			\\Computer Engineering\\ University Institute of Technology\\ 
% 	University of Ngaoundere\\  Ngaoundere (Cameroon)\\
%% 		 \normalsize{National Advanced School in Engineering (Yaounde, Cameroon)}\\
%	bertuelt@yahoo.fr\\https://orcid.org/0000-0001-8995-9522
%%	\\Skype: Bertuel TANGUE NDAWA
	\vspace{0.5cm}
 		\\\today}
				\date{ }
 				\maketitle
		\selectlanguage{english}
 			
%\hspace{-0.5cm} To my  Dad NDAWA Joseph and my Mom TCHOUMI Epse NDAWA Julienne.
 \paragraph{Abstract}
 We study a certain class  circle maps which are constant on one interval (called flat piece), and such that the degrees of the singularities at the boundary of the flat piece are different.
 In this paper, we show that if the topological conjugacy between two maps of my class is a bi-Lipschitz homeomorphism, then it is a $C^1$ diffeomorphism; that is, the  bi-Lipschitz homeomorphism class and $C^1$ diffeomorphism class of a map in our class are equivalent.

%  describe the bi-Lipschitz homeomorphism class of a map our class
% $(c_u(f), c'_u(f) )$, $( c_+(f), c'_+(f))$ and $(c_s(f), c'_s(f) )$.\vspace{0.5cm}
 
%  \textbf{\underline{Area}}:Dynamical systems and ergodic

\vspace{0.5cm}

	\section{Introduction}
A dynamical system (or a dynamic) is a triplet $(E,S_G,A)$ where $S_G$ is a semi-group, and $A: S_G\times E\longrightarrow E$ is an action of the semi-group $S_G$ on the set $E$. As an example, if $f:E\longrightarrow E$ is a map, the map $f: \N\times E\longleftarrow E$, $(n,x)\mapsto f^n(x)$ defines an action of the set $\N$ of non negative integer number on $E$. It is known that, a approximate solution  of the stability of the solar system can be given by a circle map (dynamic) $f: S^1\longleftarrow S^1$ with a flat piece. Also, some of these  maps appear naturally in the study of Cherry flows on the two dimensions torus (see \cite{MSM},\cite{Livi2013}), non-invertible (non-injective) circle continuous  maps (see \cite{Mi}), and of the dependence of the rotation interval on the parameter value for one-parameter families of circle continuous   maps (see \cite{2}). In this work, we are interested in a certain class of  weakly order preserving $C^3$ circle map with a flat piece.
 We write $S^1=\mathbb{R}/\mathbb{Z}$ for the circle, and  $\pi: \mathbb{R}\longrightarrow S^1$ is the natural projection. As a consequence, each circle map  $f$ in our class lifts as  a  unique (up to integer translation)   real continuous map $F$. The rotation number $\rho (f)$ of $f$ is 
\begin{equation*}
\rho (f):=\lim_{n\rightarrow\infty }\dfrac{F^n(x)-x}{n}(mod\:1).
\end{equation*}
This concept  was first introduced by   J. H. Poincaré, see \cite{JHP}.
The dynamic defined by a such map $f$  is more interesting when  $\rho (f)$ is irrational, see \cite{de Melo and van Strien}, pp.19-36, or Remark 4 in \cite{NTBr}. When $\rho (f)$ is irrational,  ones defines its denominators of the nearest rational approximants  as follows: 
 $$q_0=1, \;q_1=a_0, \mbox{ and } q_{n+1}=a_nq_{n}+q_{n-1} \mbox{ when } n\geq 2$$ with 
\begin{equation*}\label{i1}
\rho (f)=[a_0a_1\cdots ]:=\dfrac{1}{a_0+\dfrac{1}{a_1+\dfrac{1}{ \ddots}}}, \;a_i\in \mathbb{N}^*,\; i\in  \mathbb{N}.
\end{equation*}
%\begin{definition}
The rotation number $\rho (f)$ is said to be  Fibonacci or golden mean if
$a_n=1$ for every $n\in\mathbb{N}$. In this work, we are interested in Fibonacci circle maps; which are circle maps with Fibonacci rotation number. In addition to this assumption, we also assume that the critical exponents (the degrees of the singularities at the boundary of the flat interval) are different, and that the schwarzian derivative 
\begin{equation*}
S(f)=	\dfrac{f^{'''}}{f'}-\dfrac{3}{2}\left(\dfrac{f^{''}}{f'}\right)^2
\end{equation*} of a map  $f$ in our class is negative. Note that the proofs of our results do not directly use this assumption on the Schwarzian derivative; while, they use results in \cite{NTBr} that are proven under this assumption. 

Let $f$ and $g$ two $C^2$ circle map with a flat piece, and same rotation. Then, there is a circle homeomorphism $h$ such that $h\circ f= g\circ h$; that means, $f$ and $g$ are topological conjugate via the topological conjugacy $h$, see \cite{MSM}. That is, for a fixed rotation number, there is exactly one topological class.  
The purpose of this presentation is to characterise the $C^1$ diffeomorphism $\{h\circ f\circ h^{-1}, \;h \mbox{ is a } C^1 \mbox{ diffeomorphism }\}$ of $f$  from its topological class. 
Let $U_f$ be the flat piece of $f$. Note that  $h$ still uniquely defined on the attractor  $K_f:=\mathcal{S}^1\setminus \bigcup_{i=0}^{\infty} f^{-i}(U_f)$ of $f$, and  is arbitrarily defined outside; that is why  we are more interested in $h_{K_f}$ the restriction of $h$ on $K_f$ which we still note as $h$.

The study of circle maps with a flat part reveals many surprises. The geometry class is described by $C^{1+\beta}$ diffeomorphism (see \cite{ML,NTBr}), while in others cases as, 
circle diffeomorphisms \cite{Yo} \cite{He}, circle maps with breakpoints \cite{KK}, critical circle homeomorphisms \cite{FM}, \cite{FMP} \cite{KT}, \cite{Ya}, \cite{GMM},  \cite{AV} unimodal maps \cite{Mc}, \cite{MP}, for Kleinian groups \cite{MO},  it turned out that  the conjugations are  smooth. Also, the description of $C^1$ diffeomorphism class in this work,  Propositions 23 and 28 in \cite{NTBr} show that $C^1$ diffeomorphism class, and bi-Lipschitz homeomorphism class are equivalent. That is, if $h$ is a bi-Lipschitz homeomorphism conjugacy between two maps of our class, then $h$ is a $C^1$ diffeomorphism. As a consequence, all conjugacy between two maps of our class with properties between the one of bi-Lipschitz homeomorphism and  $C^1$ diffeomorphism are $C^1$ diffeomorphism.

Before explaining our results in more detail, we first precise the class of function,  adopt  some notations, and present  necessary previous results.  

\paragraph{The class of functions:}
We write $S^1$  as an interval $[a, 1]$ where we identify $a$ with 1.
We denote by $ \widetilde{\mathcal{W}}^X_{[1]} $ the set of Fibonacci $C^3$ circle maps with a flat piece, different critical exponents, and negative Schwartzian derivative.  We consider $ \widetilde{\mathcal{W}}^X_{[1]} $ as a subset of the Cartesian product $\Sigma^{X}\times(Diff^3([0,1]))^3$ where 
\begin{equation*}
	\Sigma^{X}=\{(x_1,x_2,x_3,x_4,s)\in \mathbb{R}^5\,|\;
	x_1<0<x_2,\;x_3<x_4<1,\;0<s<1\}.
\end{equation*}
%Let $f\in \widetilde{\mathcal{W}}^X_{[1]}$. 

%then there exists $(\ell_1, \ell_2)\in(1,2)^2$, $\ell_1\neq \ell_2$, 

%We consider the class $\mathscr{L}^{X} $ described as follows:
%
%We fix $\ell_1, \ell_2>1$ and denote:\\
%-by  $\Sigma^{X}$ the set
%\begin{equation*}
%	\Sigma^{X}=\{(x_1,x_2,x_3,x_4,s)\in \mathbb{R}^5\,|\;
%	x_1<0<x_2,\;x_3<x_4<1,\;0<s<1\},
%\end{equation*}
%- for $r\in\mathbb{N}$, by $Diff^r([0,1])$ the space of $c^r$ orientation
%preserving
%diffeomorphisms of $[0,1]$. 
%
%The space of $C^3$ circle maps with a flat interval is denoted by:
%\begin{equation*}
%	\mathscr{L}^{X}=\Sigma^{X}\times(Diff^3([0,1]))^3.
%\end{equation*}
A point
\begin{equation*}
	f:=(x_1,x_2,x_3,x_4,s,\varphi,\varphi^l,\varphi^r)
	\in \widetilde{\mathcal{W}}^X_{[1]}
\end{equation*}
is 
$f:[x_1,1]\longrightarrow[x_1,1]$ (by identifying $x_1$ with 1),
\begin{equation*}
x\mapsto \begin{cases}
		\begin{array}{l}f_-(x)= \begin{array}{lcl}f_1(x)=(1-x_2)q_s\circ\varphi\left(\dfrac{x_1-x}{x_1}\right) +x_2 & \mbox{if} & x\in [x_1,0]\end{array}\\
			f_+(x)= \begin{cases} \begin{array}{lcl}f_2(x)=
					x_1\left(\varphi^{l}\left(\dfrac{x_3-x}{x_3}\right)\right)^{\ell_1}  & \mbox{if} & x\in [0,x_3]\\f_3(x)=
					0 & \mbox{if} & x\in [x_3,x_4]\\f_4(x)=
					x_2\left(\varphi^{r}\left(\dfrac{x-x_4}{1-x_4}\right)\right)^{\ell_2 }&  \mbox{if} & x\in [x_4,1]\end{array}
			\end{cases}
		\end{array}
	\end{cases}
\end{equation*}
where  $(\ell_1, \ell_2)\in(1,2)^2$, $\ell_1\neq \ell_2$ is the critical exponents of $f$, and 
$q_s:[0,1]\longrightarrow [0,1] $  is the map
\begin{equation*}
	q_s (x)= \dfrac{[(1-s)x+s]^{\ell_2}-s^{\ell_2} }{1-s^{\ell_2}}.
\end{equation*}

%
%\begin{figure}[H]	
%	\centering
%	%	\begin{turn}{90}
%		\setlength{\unitlength}{7cm}
%		\begin{picture}(1, 1)
%			\put(0,0){\line(0,1){1}}
%			\put(0,0){\line(1,0){1}}
%			\put(1,0){\line(0,1){1}}
%			\put(0,1){\line(1,0){1}}
%			\put(0,0){\line(1,1){1}}
%			\put(0,0.55){\line(1,0){1}} % ligne horizontale sur x_2
%			\put(0.35,0){\line(0,1){1}} % ligne verticale sur 0
%			
%			\multiput(0.55,0)(0,0.02825){20}
%			{\line(0,1){0.013755}} % ligne verticale sur x_2
%			
%			%	\put(0.55,0.55){\line(0,-1){0.007}}
%			\qbezier(0.35,0)(0.5,0.29)(0.65,0.35)
%			\put(0.65,0.35){\color{red}{\line(1,0){0.12}}} % partie plate
%			\qbezier(0.77,0.35)(0.95,0.4)(1,0.55)
%			\put(0.052,0.78){$f^{q_0},\;f_- $}
%			\put(0.72,0.425){$f^{q_1},\;f_+ $}
%			\multiput(0,0.35)(0.0275,0){24}
%			{\line(1,0){0.013755}}
%			\multiput(0.65,0)(0,0.02795){13}
%			{\line(0,1){0.013755}} % ligne horizontale sur x_3
%			\multiput(0.77,0)(0,0.02795){13}
%			{\line(0,1){0.013755}} % ligne horizontale sur x_4
%			\qbezier(0,0.55)(0.18,0.6)(0.35,1)
%			\put(-0.01,-0.04){$x_1$}
%			\put(0.34,-0.0475){$0$}
%			\put(0.54,-0.04){$x_2$}
%			%	\put(0.64,0.37){$x^{\ell_1}$}
%			%	\put(0.815,0.37){$x^{\ell_2}$}
%			\put(0.64,-0.04){$x_3$}
%			\put(0.76,-0.04){$x_4$}
%			\put(0.985,-0.0475){$1$}
%			\put(-0.041,0.96){$1$}
%		\end{picture}
%		%	\end{turn}{90}	
%	\vspace{0.5cm}	
%	\caption[cite]{A map $f$ in $\mathscr{L}^{X}$ with $x_2<x_3$} % title of the Figure 
%	\label{a map 1}    % label to refer figure in text 
%\end{figure}

\paragraph{Notations.}
\subparagraph{Common notations.}
Let $f\in \widetilde{\mathcal{W}}^X$ with $U$ as its flat piece.
\begin{enumerate}
	\item Let  $i\in\mathbb{Z}$. We will simply write $\underline{i} $ instead of $f^i(U)$. For example, $\underline{0}=U $. We write the length of the interval $\underline{i} $ as $|\underline{i}|$. So,  $|\underline{i}|  =0$ if $i>0$.  The number $|(\underline{i},\underline{j} )|$ stands for  the distance between the closest endpoints
	of these two intervals, and  $|[\underline{i},\underline{j} )|$ is $|\underline{i}|+|(\underline{i},\underline{j} )|$. 
	\item Let $I$ be an interval. The interval   $ \bar{I} $ is the closure of $I$.
	\item The scaling ratios $\alpha_n, \,n\in\mathbb{N}$ are defined by 
	\begin{equation*}
		\alpha_n:=\dfrac{|(f^{-q_n}(U), U)|}{|(f^{-q_n}(U), U)|+|f^{-q_n}(U)|}=\dfrac{|(\underline{-q_n}, \underline{0})|}{|[\underline{-q_n}
			, \underline{0})|}, \,n\in\mathbb{N}.
	\end{equation*}
	
%	\item 
%	For any sequence $\Gamma_n$  and for any  real $d$ we have:
%	\begin{equation*}
%		\Gamma_n^{d(\ell_1,\ell_2)}:=\begin{cases}\begin{array}{lcl}
%				\Gamma_n^{d\ell_1} & \mbox{if}& n\equiv 0[2]\\
%				\Gamma_n^{d\ell_2} &\mbox{if} & n\equiv 1[2]
%			\end{array}
%		\end{cases}  \Gamma_n^{d(\frac{1}{\ell_1},\frac{1}{\ell_2})}:=\begin{cases}\begin{array}{lcl}
%				\Gamma_n^{d\frac{1}{\ell_1}} & \mbox{if}& n\equiv 0[2]\\
%				\Gamma_n^{d\frac{1}{\ell_2}} &\mbox{if} & n\equiv 1[2].
%			\end{array}
%		\end{cases}
%	\end{equation*}	
%	For example, $\alpha_1^{\ell_{1},\ell_{2}}= \alpha_1^{\ell_{2}}$.
	\item Let $x_n$ and $y_n$ be two sequences of positive numbers. We say that $x_n$ is of the order of $y_n$ if there exists a uniform positive constant $k$ such that, for $n$ big enough $x_n<ky_n$. We will use the notation
	\begin{equation*}
		x_n= O(y_n).
	\end{equation*}
\end{enumerate}

\subparagraph{Parameters  frequently used.} 
Let $f\in \widetilde{\mathcal{W}}^X_{[1]}$.
%	 We put together the parameters which are frequently used in the main parts of this paper.
\begin{enumerate}
	\item The geometrical (rigidity) characteristics $(c_u(f), c'_u(f) )$, $( c_+(f), c'_+(f))$ and $(c_s(f), c'_s(f) )$ are defined in Proposition~\ref{super formular}. Moreover,
	\begin{enumerate}
		\item[-]  $c^{*}_{\iota}(f)=c^{}_{\iota}(f) \mbox{ or }c'_{\iota}(f),\; \iota=s,u,+ $
%		 or Notation~\ref{notation cus+}; 
		\item[-] $\underline{c}_u(f)=\min\{c_u(f),c'_u(f)\}$ and  $\overline{c}_u(f)=\max\{c_u(f),c'_u(f)\}$.		
	\end{enumerate}
	\item For every $n\in\mathbb{N}$, $S_{i,n}$, $y_{i,n};\,i=1,2,3,4,5$ are defined   in   (\ref{change variable x S}) and  (\ref{change variable S y}) respectively.
	\item The vectors $e_i^{\iota}; \,i=2,3,4,5;\,\iota=s,u,+$ are defined in  Proposition~\ref{super formular}. 
	
\end{enumerate}

Let $(1,2)^2_{\neq}:=\{(\ell_1,\ell_2)\in(1,2)^2, \;\ell_1\neq\ell_2\in(1,2)\}$.
 The main result of this presentation can be formulated as follows. 
\paragraph{Main result}
Let  $(\ell_{1},\ell_{2})\in(1,2)^2_{\neq}$. 	Let   $f,g\in \widetilde{\mathcal{W}}^X_{[1]}$  with  $(\ell_{1},\ell_{2})$ as their  critical exponents. If  $h$ is the topological conjugation between $f$ and $g$,  then 
\begin{align*}
	h \mbox{ is a } C^1  \mbox{ diffeomorphism  }\Longleftrightarrow& c^*_u(f)=c^*_u(g),  c_+(f)+c'_+(f)=c_+(g)+c'_+(g).\\
\end{align*}

%\begin{prop}\label{preimage and gaps adjacent}[\cite{GJSTV}, Proposition 2, P.606]
%	Let $n\geq 1$ and                       
%	$	\wp_n:=\{ \underline{-i}; \; 0\leq i\leq  q_{n+2}-1    \} .$
%	Let $ \underline{-i} \in\wp$ and 
%	$[\underline{-i},\underline{-j})$ be one of the
%	gaps adjacent to $ \underline{-i} $, then \begin{equation*}   \dfrac{|\underline{-i}|}{|[\underline{-i},\underline{-j})|}
%	\end{equation*}	 is bounded away from zero by a constant that
%	does not depend on $n$, $ \underline{-i} $, or $\underline{-j}$.
%\end{prop}

\section{Renormalization}

 \begin{rem}~\label{critical exponent and  renormalization}
	 For every  sequence $\eta_{n}$ in the following, we have:   The terms $\eta_{2n}$, $n\in\N$ depend on $(\ell_1, \ell_2)$ by a function 
		$\Psi_{2n}$ if and only if  the terms $\eta_{2n+1}$, $n\in\N$ depend on $(\ell_2, \ell_1)$ by the same function. Therefore, a statement or  proof presented for $n$  odd  deduces by himself the case $n$ even  and vice-versa. 
\end{rem}
%\subsection{Renormalization}
Let  $(\ell_{1},\ell_{2})\in(1,2)^2_{\neq}$. Let $f\in  \widetilde{\mathcal{W}}^X_{[1]}$ with $(\ell_{1},\ell_{2})$ as their  critical exponents. 
Let  $x_{1,1}=x_2/x_1$, and let $h: [x_1,x_2)\longrightarrow [x_{1,1},1)$ be the map $x\longmapsto {x}/{x_1} $. The first renormalization  $\mathcal{R} f:  [x_{1,1},1)\longrightarrow [x_{1,1},1)$ is the map
%\begin{equation}\label{renormalization definition}
%x\mapsto	\mathcal{R} f (x)=\begin{cases}
%		\begin{array}{lcl}
%			\mathcal{R} f_-(x)=h\circ f_+\circ h^{-1}(x) &\mbox{if} & x \in [x_{1,1},0)\\
%			\mathcal{R} f_+(x)=h\circ f_+^{a_1}\circ f_-\circ h^{-1}(x) &\mbox{if} & x \in (0,1)
%		\end{array}
%	\end{cases} 
%\end{equation} 
%More precisely,
\begin{equation*}\label{renormalization definition explicit}
x\mapsto	\mathcal{R} f (x)=\begin{cases}
		\begin{array}{lcl}
			\mathcal{R} f_-(x)=h\circ f_2\circ h^{-1}(x) &\mbox{if} & x \in [x_{1,1},0)\\
			\mathcal{R} f_+(x)=h\circ f_2\circ f_1\circ h^{-1}(x) &\mbox{if} & x \in (0,1]
		\end{array}
	\end{cases}
\end{equation*}
Observe that, for every $n\geq2$, $	\mathcal{R}^n f:=	\mathcal{R}^{n-1}\circ f	\mathcal{R} f$ is defined, and belongs into $\widetilde{\mathcal{W}}^X_{[1]}$. That is, $f$ is infinitely renormalizable. We write $\mathcal{R}^n f:=(x_{1,n} ,x_{2,n} ,x_{3,n} ,x_{4,n} ,s_n,\varphi_n,\varphi^l_n,\varphi^r_1)$.

The set $\Sigma^{X} $ can be redefined as $\Sigma^{S} $ or $\Sigma^{Y}$  as follows:
%\paragraph{changes of variables}
%\subsubsection{Changes of variables
\begin{description}
	\item[] $(X)\longrightarrow (S)$. Let
	\begin{equation} \label{change variable S x}\begin{array}{c}
			S_{1,n}=\dfrac{x_{3,n}-x_{2,n}}{x_{3,n}},\;S_{2,n}= \dfrac{1-x_{4,n}
			}{1-x_{2,n}},\\\;S_{3,n}=\dfrac{x_{3,n}}{1-x_{4,n}},\;S_{4,n}=-\dfrac{x_{2,n}}{x_{1,n}},\;S_{5,n}=s_n.
		\end{array}
	\end{equation}
	That is,
	\begin{equation}\label{change variable x S}
		\begin{cases}
			\begin{array}{l}x_{1,n}= \dfrac{S_{3,n}(1-S_{1,n})S_{2,n}}{(1+S_{3,n}(1-S_{1,n})S_{2,n})S_{4,n}} \\
				x_{2,n}= \dfrac{S_{3,n}(1-S_{1,n})S_{2,n}}{1+S_{3,n}(1-S_{1,n})S_{2,n}}\\
				x_{3,n}= \dfrac{S_{3,n}S_{2,n}}{1+S_{3,n}(1-S_{1,n})S_{2,n}}\\
				x_{4,n}=1-\dfrac{S_{2,n}}{1+S_{3,n}(1-S_1)S_{2,n}}
			\end{array}
		\end{cases}
	\end{equation}
	
	\item[]$(S) \longrightarrow (Y)$. Let
	\begin{equation}\label{change variable S y}
		\begin{array}{c}
			y_{1,n}=S_{1,n},\;y_{2,n}=\ln S_{2,n},\;y_{3,n}=\ln S_{3,n},\\y_{4,n}=\ln S_{4,n},\;y_{5,n}=\ln S_{5,n}.
		\end{array}
	\end{equation}
\end{description}
For more detail, see Subsection 4.1 in \cite{NTBr}.

%The rest of the results in this section come from 

%\begin{prop}\label{asymptotic distortion}
%	Let $ f\in \widetilde{\mathcal{W}}^X_{[1]} $.  For every $n$,
%	\begin{equation*}
%		dist(\varphi_{n}^l)=O\left(\alpha_{n-1}^{\frac{1}{\ell_{1}}, \frac{1}{\ell_{2}}}\right),\;	dist(\varphi_{n})=O\left(\alpha_{n-2}^{\frac{1}{\ell_{1}}, \frac{1}{\ell_{2}}}\right),\;
%		dist(\varphi_{n}^r)=O\left(\alpha_{n}^{\frac{1}{\ell_{1}}, \frac{1}{\ell_{2}}}\right)
%	\end{equation*}
%\end{prop}

%\begin{cor}\label{corollary renormalization on S}
%	Let	$ f\in \widetilde{\mathcal{W}}^X_{[1]} $, $n\in\N$,  and let
%	\begin{equation*}
%		R^n(f)=\left(S_{1,n},
%		S_{2,n},S_{3,n},S_{4,n},S_{5,n},\varphi_n,\varphi_n^l,\varphi_n^r\right).
%	\end{equation*}
%	Then,
%	\begin{equation*}
%		\frac{\ell_{2}S_{1,2n}^{\ell_1}}{S_{2,2n}}= 1+O\left(\alpha_{2n-1}^{\frac{1}{ \ell_{2}}}\right)\quad \mbox{  and  } \quad \frac{\ell_{1}S_{1,2n+1}^{\ell_2}}{S_{2,2n+1}}= 1+O\left(\alpha_{2n}^{\frac{1}{ \ell_{1}}}\right).
%	\end{equation*}
%\end{cor}

\begin{prop}[Proposition 12, \cite{NTBr}]\label{super formular}
	Let $(\eg,\ed)\in(1,2)^2$. Then, there exists $\lambda_u>1,\;\lambda_s\in(0,1)\;
	E^u,\;E^s$, $E^+,$  $w_{fix}\in\R^4,$ such that the following
	holds. Given $ f\in \widetilde{\mathcal{W}}^X_{[1]} $ with critical exponents  $(\eg,\ed) $, there exists $c_u(f), c'_u(f)<0$, $c_s(f), c'_s(f)$, $c_+(f)$ and $c'_+(f)$ such that, for
	all $n:=2p_n\in\N^*$; $p_n\in\N$,
	\begin{equation*}
		w_{n}(f)=c_u(f) \lambda_u^{p_n}E^u+c_s(f)
		\lambda_s^{p_n}E^s+c_+(f) E^+ + w_{fix}+O(e,c_u(f),\lambda_u,n)
	\end{equation*}
	and 
	\begin{equation*}
		w_{n+1}(f)=c'_u(f) \lambda_u^{p_n}E^u+c'_s(f)
		\lambda_s^{p_n}E^s+c'_+(f) E^+ + w_{fix}+O(e,c'_u(f),\lambda_u,n);
	\end{equation*}
	where, $O(e,c_u(f),\lambda_u,n)$ is the vector whose components are  
	\begin{equation*}
		O\left((e^{c_u(f) \lambda_u^{p_{n-4}}})^{1/\overline{\ell}}\right);
	\end{equation*}
	with,
	\begin{equation*}
		\overline{\ell}:=\max\{\ell_1,\ell_2\}.
	\end{equation*}
	And, we have also
	\begin{equation*}\begin{array}{c}
			dist(\varphi_n)=O\left(e^{\frac{c_u(f) \lambda_u^{p_{n-2}}}{\ell_{1}}}\right),\\
			dist(\varphi_n^l)=O\left(e^{\frac{c_u(f) \lambda_u^{p_{n-1}}}{\ell_{2}}}\right),\\
			dist(\varphi_n^r)=O\left(e^{\frac{c_u(f)
					\lambda_u^{p_{n}}}{\ell_{1}}}\right).
	\end{array}\end{equation*}
\end{prop}

%\begin{lem}\label{alpha_n lambda u}
%	Let $ f\in \widetilde{\mathcal{W}}^X_{[1]} $, then,
%	\begin{equation*}
%		\alpha_{2p_n}=O\left(e^{c_u(f) \lambda_u^{{p_n}}(e_2^u+e_3^u)}\right),\; 	\alpha_{2p_n+1}=O\left(e^{c'_u(f) \lambda_u^{{p_n}}(e_2^u+e_3^u)}\right)\; \mbox{and } c'_u(f),\,c_u(f)<0.\end{equation*}
%\end{lem}

%
\begin{lem}[Lemma 20, \cite{NTBr}]\label{xn asymptotic coordinate}
	Let $(\ell_1,\ell_{2})\in(1,2)^2$, and let $f\in \widetilde{\mathcal{W}}^X_{[1]} $ with critical exponent  $(\ell_1,\ell_{2})$. Then for $n:=2p_n$,
	\begin{equation*}
		\begin{array}{rcl}-x_{1,n} &=&e^{c_u(f)\lambda_u^{p_{n}}(e^u_2+e^u_3-e^u_4) +c_s(f)\lambda_s^{p_{n}}(e^s_2+e^s_3-e^s_4)-c_++ 0((e^{c_u(f)\lambda_u^{p_{n-4}}(e^u_2+e^u_3)})^{1/\overline{\ell}})}\\
			x_{2,n} &=& e^{c_u(f)\lambda_u^{p_{n}}(e^u_2+e^u_3) +c_s(f)\lambda_s^{p_{n}}(e^s_2+e^s_3)+ 0((e^{c_u(f)\lambda_u^{p_{n}}(e^u_2+e^u_3)})^{1/\overline{\ell}})}\\
			x_{3,n} &=& e^{c_u(f)\lambda_u^{p_{n}}(e^u_2+e^u_3) +c_s(f)\lambda_s^{p_{n}}(e^s_2+e^s_3)+ 0((e^{c_u(f)\lambda_u^{p_{n-4}}(e^u_2+e^u_3)})^{1/\overline{\ell}})}\\
			1-x_{4,n}&=& e^{c_u(f)\lambda_u^{p_{n}}e^u_2 +c_s(f)\lambda_s^{p_{n}}e^s_2+ 0((e^{c_u(f)\lambda_u^{p_{n-4}}(e^u_2+e^u_3)})^{1/\overline{\ell}})}
		\end{array}
	\end{equation*}
\end{lem}

%\begin{lem}	\label{asymptotic coordonat in x}
%	Let $ f\in\widetilde{\mathcal{W}}^X_{[1]}$, then
%	\begin{enumerate}
%		\item $	\dfrac{x_{2,n+1}}{x_{1,n}}=O(\alpha_{n+1})$,
%		\item $\dfrac{x_{3,n+1}}{x_{1,n}}=O(\alpha_{n+1})$,
%		\item $\dfrac{x_{1,n}-x_{4,n+1}}{x_{1,n}}=O(\alpha_{n})$,
%		\item $S_{1,n}=\dfrac{x_{3,n}-x_{2,n}}{x_{3,n}}=O(\alpha_{n+1})$.
%		\item $s_{n}=S_{5,n}=O(\alpha_{n})$.
%	\end{enumerate}
%\end{lem}

%

\section{$C^1$ diffeomorphism class}
Throughout this section,  $f,g\in\widetilde{\mathcal{W}}^X_{[1]}$ with critical exponent $(\ell_{1},\ell_{2})\in(1,2)^2_{\neq}$, $U_f$ and $U_g$ as their flat piece respectively. The map $h$ is the topological conjugation between $f$ and $g$.

For $n$ even large enough, we obtain
\begin{equation}\label{R1}
	f(u_f)-f^{q_{n+1}+1}(u_f)=\prod_{k\leq n}^{}S_{4,k}(f)\sim e^{c_u(f)e^u_4(f)\lambda_u^{p_{n}}(f)+ \frac{n}{2}(c_+(f)+c'_+(f))},
\end{equation}
\begin{equation}\label{R1'}
	g(u_g)-g^{q_{n+1}+1}(u_g)=\prod_{k\leq n}^{}S_{4,k}(g)\sim e^{c_u(g)e^u_4(g)\lambda_u^{p_{n}}(g)+ \frac{n}{2}(c_+(g)+c'_+(g))}.
\end{equation}
Where we use Proposition~\ref{super formular}.

\begin{prop}\label{h bi-lipsch, c1 diffeo}
	Let $(\ell_{1},\ell_{2})\in(1,2)^2$. 
	If $f,g\in\widetilde{\mathcal{W}}^X_{[1]}$ with critical exponent $(\ell_{1},\ell_{2})$, and if $h$ is the  topological conjugation between $f$ and $g$ then
	\begin{equation*}
		h \mbox{ is a } C^1 \mbox{diffeomorphism} \Longrightarrow  c^*_u(f)=c^*_u(g),  c_+(f)+c'_+(f)=c_+(g)+c'_+(g).
	\end{equation*}
\end{prop}				
\begin{proof}
	The proposition stems directly from \textbf{(\ref{R1})} and \textbf{(\ref{R1'})}.
\end{proof}

In the following, we use the below notations. For all $n\in\N$,
\begin{description}
	\item[$\bullet$] $A_n(f)= (f(U_f)  , f^{q_{n}+1}(U_f))$;
	\item[$\bullet$]   $B_n(f)= (f^{-q_{n}+1}(U_f),f(U_f)   )$;
	\item[$\bullet$]  $C_n(f)=f^{-q_{n}+1}(U_f)$;
	\item[$\bullet$]  $D_n(f)= ( f^{q_{n-1}+1}(U_f),f^{-q_{n}+1}(U_f))$. 
\end{description}
And their iterates
\begin{description}
	\item[$\bullet$] $A^{i}_n(f)= f^{i}(A_n(f))$ for $0\leq i< q_{n-1} $;
	\item[$\bullet$]    $B^{i}_n(f)= f^{i}(B_n(f))$ for $0\leq i< q_{n} $;
	\item[$\bullet$]  $C^{i}_n(f)= f^{i}(C_n(f))$ for $0\leq i< q_{n}$;
	\item[$\bullet$]   $D^{i}_n(f)= f^{i}(D_n(f))$ for $0\leq i< q_{n}$. 
\end{description}
Observe that for all $n\in\N$, $
\mathcal{P}_{n}=\{A^{i}_n(f), B^{j}_n(f), C^{j}_n(f), D^{j}_n(f)|\; 0\leq i< q_{n-1},\;0\leq j< q_{n} \} $ together with the boundary points of intervals  are the $n$th partition of $Dom(f)$ the domain of $f$. Since that $h(U_f)=U_g$, it is easy to see that 

\begin{description}
	\item[$\bullet$] $h(A^{i}_n(f))= A^{i}_n(g)$ for $0\leq i< q_{n-1} $;
	\item[$\bullet$]  $h(B^{i}_n(f))= B^{i}_n(g)$ for $0\leq i< q_{n} $;
	\item[$\bullet$]   $h(C^{i}_n(f))= C^{i}_n(g)$ for $0\leq i< q_{n}$;
	\item[$\bullet$]  $h(D^{i}_n(f))= D^{i}_n(g)$ for $0\leq i< q_{n}$. 
\end{description}

Let $n\in\N$. Let $Dh_n: [0,1]\longrightarrow \R^+ $ the function defined by:
\begin{equation*}
	Dh_n(x)= \dfrac{|h(I)|}{|I|};
\end{equation*}
with, $I\in\mathcal{P}_n $ and $x\in  \stackrel{\circ}{I} $. Observe that, for all $T=\cup I_i$; $I_i\in \mathcal{P}_n $ and for all $m\geq n$, 
\begin{equation}\label{R3}
	h(T)= \int_{T}^{}Dh_m.
\end{equation}
We have the following result.  
\begin{lem}[Lemma 26, \cite{NTBr}]\label{Dhn cauchy}
Let $f,g\in\widetilde{\mathcal{W}}^X_{[1]}$. If $c^*_u(f)=c^*_u(g)$ and $c_+(f)+c'_+(f)=c_+(g)+c'_+(g)$, then 
	\begin{equation*}
		\log\dfrac{Dh_{n+1}}{Dh_{n}}=O\left(\alpha_{n-2}^{\frac{1}{\ell_{1}},\frac{1}{\ell_{2}}}+|c^*_s(g)-c^*_s(f)|\lambda_s^{p_n}\right). 
	\end{equation*}
\end{lem}

Let $I$  be an interval of $\mathcal{P}_n$. Observe that every boundary point of $I$ is in  the orbit of the critical point $o(f(U))$. As a consequence,  $Dh_n$ is well defined for $x\not\in o(f(U)$. Lemma~\ref{Dhn cauchy} implies that

%Note that, every boundary point of an interval in $\mathcal{P}_n$ is in the orbit of the critical point $o(f(U))$. So, if $x\not\in o(f(U) $ then $Dh_n$ is well defined. \textbf{Lemma~\ref{Dhn cauchy}} implies that
\begin{equation*}
	D(x)=\lim\limits_{n\rightarrow\infty}Dh_n(x)
\end{equation*}
exists and, \begin{equation}\label{R00}0<\inf_xD(x)<\sup_xD(x)<\infty.\end{equation}

\begin{lem}\label{Dhn go to D}
	Let $f,g\in\widetilde{\mathcal{W}}^X_{[1]}$ with critical exponent
	$(\ell_{1},\ell_{2})\in(1,2)^2_{\neq}$.  If  If $c^*_u(f)=c^*_u(g)$, and $c_+(f)+c'_+(f)=c_+(g)+c'_+(g)$, then 
	\begin{equation*}
		D: Dom(f)\setminus o(f(U))\longrightarrow (0,\infty)
	\end{equation*}
	is continuous. In particular, for all $n\in\N$,
	\begin{equation*}
		\log\dfrac{D(x)}{Dh_n(x)}=O(\alpha_{n-2}^{\frac{1}{\ell_{1}},\frac{1}{\ell_{2}}}+|c^*_s(g)-c^*_s(f)|\lambda_s^{p_n}) 
	\end{equation*}
\end{lem}
\begin{proof}
Let $n_0\in\N$. From \textbf{Lemma~\ref{Dhn cauchy}}, we have 
	\begin{align*}
		\left|\log\dfrac{D(x)}{Dh_{n_0}(x)}\right| &\leq \sum\limits_{k\geq 0}^{}\log\dfrac{Dh_{n_0+k+1}(x)}{Dh_{n_0+k}(x)}\\
		&\leq \sum\limits_{k\geq 0}^{} O(\alpha_{k-2}^{\frac{1}{\ell_{1}},\frac{1}{\ell_{2}}}+|c^*_s(g)-c^*_s(f)|\lambda_s^{p_k})\\ &\leq
		O(\alpha_{n_0-2}^{\frac{1}{\ell_{1}},\frac{1}{\ell_{2}}}+|c^*_s(g)-c^*_s(f)|\lambda_s^{p_{n_0}}).
	\end{align*}
	Therefore, if $x,y\in I\in\mathcal{P}_{n_0} $
	\begin{align*}
		\left|\log\dfrac{D(x)}{Dh_(y)}\right| &\leq \left|\log\dfrac{D(x)}{Dh_{n_0}(x)}\right| +\left| \log\dfrac{D(y)}{Dh_{n_0}(x)}\right| \\
		&\leq
		O(\alpha_{n_0-2}^{\frac{1}{\ell_{1}},\frac{1}{\ell_{2}}}+|c^*_s(g)-c^*_s(f)|\lambda_s^{p_{n_0}});
	\end{align*}
	which means that $D$ is continuous. 
\end{proof}

Let $T$ be an interval in $Dom(f)$. Then by formula \textbf{(\ref{R3})},
\begin{equation}
	|h(T)|=\lim\limits_{n\rightarrow \infty}|h_n(T)|=\lim\limits_{n\rightarrow \infty}\int_{T}Dh_n\leq \int_T\sup_xD(x)\leq \sup_xD(x)|T|.\label{R4}
\end{equation}
%Also,
%\begin{equation}
%	|h(T)|=\lim\limits_{n\rightarrow \infty}|h_n(T)|=\lim\limits_{n\rightarrow \infty}\int_{T}Dh_n\geq \int_T\inf_xD(x)\geq \inf_xD(x)|T|.\label{R4'}
%\end{equation}

\begin{prop}\label{h homeo diffeo}
	Let $f,g\in\widetilde{\mathcal{W}}^X_{[1]}$ with  $(\ell_{1},\ell_{2})\in(1,2)^2_{\neq}$ as its critical exponent, and let $h$ the topological conjugation between $f$ and $g$. Then
	\begin{equation*}
		h \mbox{ is } C^1  \mbox{ diffeomorphism } \Longleftarrow  c^*_u(f)=c^*_u(g), c_+(f)+c'_+(f)=c_+(g)+c'_+(g).
	\end{equation*}
\end{prop}
\begin{proof}
Observe that by \textbf{(\ref{R4})}, $h$ is differential on $Dom(f)\setminus o(f(U))$  with $Dh=h$. The rest of the proof consists of extending $D$ in a continuous function on $Dom(f)$ under condition $c^*_u(f)=c^*_u(g)$  and $c_+(f)+c'_+(f)=c_+(g)+c'_+(g)$. are going to show under this condition that, for every 
% then by  inequalities \textbf{(\ref{R4})},   \textbf{Lemma~\ref{Dhn go to D}} and   inequalities in \textbf{(\ref{R00})}, $h$  is a $C^1$ diffeomorphism.
%	
%	
%	
%Let us suppose that  $c^*_s(f)=c^*_s(g)$, $c^*_u(f)=c^*_u(g)$ and  $c_+(f)+c'_+(f)=c_+(g)+c'_+(g)$.
%
%We are going to  show first that, under conditions $c^*_u(f)=c^*_u(g)$ and  $c_+(f)+c'_+(f)=c_+(g)+c'_+(g)$,  $h$ is $C^1$ diffeomorphism.
%
%Since by (\ref{R4}) $h$ is differentiable with $Dh(x)=D(x)$, for all $x$, then it remains to prove that $D$ can be extended to a continuous function. This is possible if and only if for every 
$k\geq 0$ 
\begin{equation*}
	\lim\limits_{x\rightarrow c_k^-}D(x)=\lim\limits_{x\rightarrow c_k^+}D(x)
\end{equation*}
where $c_k=f^k(f(U))$. 

Let  $k\geq 0$,  and let $n\in\mathbb{N}$ big enough such that $q_{n-1}>k$.We have

\begin{equation*}
	D_-(c_k):=\lim\limits_{x\rightarrow c_k^-}D(x)=\lim\limits_{n\rightarrow \infty}\dfrac{|B^k_{2n}(g)|}{B^k_{2n}(f)|},
\end{equation*}
and 
\begin{equation*}
	D_+(c_k):=\lim\limits_{x\rightarrow c_k^+}D(x)=\lim\limits_{n\rightarrow \infty}\dfrac{|A^k_{2n}(g)|}{A^k_{2n}(f)|}.
\end{equation*}
By Proposition~\ref{super formular}, and  Lemma~\ref{xn asymptotic coordinate}, we get
\begin{align*}
	\dfrac{D_-(c_k)}{D_+(c_k)}=\lim\limits_{n\rightarrow \infty} \dfrac{\dfrac{|B^{k}_{2n}(g)|}{|A^{k}_{2n}(g)|}}{\dfrac{|B^{k}_{2n}(f)|}{|A^{k}_{2n}(f)|}}
	=\lim\limits_{n\rightarrow \infty} \dfrac{\dfrac{|B_{2n}(g)|}{|A_{2n}(g)|}}{\dfrac{|B_{2n}(f)|}{|A_{2n}(f)|}}\vspace{0.2cm}
	=\lim\limits_{n\rightarrow \infty} \dfrac{\dfrac{x_{3,2n}(g)}{-x_{1,2n}(g)}}{\dfrac{x_{3,2n}(f)}{-x_{1,2n}(f)}}=\lim\limits_{n\rightarrow \infty}e^{\lambda_s^n}=1.
\end{align*}
\end{proof}

 By combining  Propositions 23 and 28 in \cite{NTBr}, Propositions~\ref{h bi-lipsch, c1 diffeo} and \ref{h homeo diffeo}, for Fibonacci circle maps with a flat piece, negative Schwartzian derivative, and different critical exponents, bi-Lipschitz  and $C^1$ diffeomorphism class  are equivalent. 
%	$$P_1: XxY\longrightarrow X , (x,y)\mapsto x$$

\end{document}